\documentclass[a4paper]{amsart}
\usepackage[utf8]{inputenc}
\usepackage{amsmath,amssymb, comment}
\usepackage{graphicx,multirow,array,enumerate,wasysym}
\usepackage[left=3cm,right=3cm,bottom=3cm,top=3cm]{geometry}
\usepackage{amsthm}
\usepackage{natbib}
\usepackage{tikz}
\usepackage{microtype}
\usepackage{mathtools}
\usepackage{hyperref}

\bibpunct{[}{]}{,}{n}{}{;}

\newtheorem{theorem}{Theorem}[section]

\newtheorem{conjecture}[theorem]{Conjecture}
\newtheorem{corollary}[theorem]{Corollary}
\theoremstyle{definition}

\newtheorem{example}[theorem]{Example}
\newtheorem{remark}[theorem]{Remark}

\newcommand{\SZ}{\mathbb{Z}}                    
\newcommand{\SC}{\mathbb{C}}                    
\newcommand{\SP}{\mathbb{P}}                    %
\newcommand{\CF}{\mathcal{F}}                    

\newcommand{\frakg}{\mathfrak{g}}
\newcommand{\frakheis}{\mathfrak{heis}}
\newcommand{\ra}[1]{\kern-1.5ex\xrightarrow{\ \ #1\ \ }\phantom{}\kern-1.5ex}
\newcommand{\ras}[1]{\kern-1.5ex\xrightarrow{\ \ \smash{#1}\ \ }\phantom{}\kern-1.5ex}

\title[Euler characteristics of Hilbert schemes of points on surfaces]{Euler characteristics of Hilbert schemes of points\\ on surfaces with simple singularities}
\author{Ádám Gyenge}
\address{Alfréd Rényi Institute of Mathematics, Hungarian Academy of Sciences, Budapest, Hungary}
\email{gyenge.adam@renyi.mta.hu}

\author{András Némethi}
\address{Alfréd Rényi Institute of Mathematics, Hungarian Academy of Sciences, Budapest, Hungary}
\email{nemethi.andras@renyi.mta.hu}

\author{Balázs Szendrői}
\address{Mathematical Institute, University of Oxford}
\email{szendroi@maths.ox.ac.uk}

\begin{document}

\begin{abstract}
This is an announcement of conjectures and results concerning the generating
series of Euler characteristics of Hilbert schemes of points on
surfaces with simple (Kleinian) singularities. For a quotient surface
$\SC^2/G$ with $G<{\mathrm SL}(2, \SC)$ a finite subgroup, we conjecture a
formula for this generating series in terms of Lie-theoretic data, 
which is compatible with existing results for type $A$
singularities. We announce a proof of our conjecture for singularities of type $D$. 
The generating series in our conjecture can be seen as a specialized
character of the basic representation of the corresponding (extended)
affine Lie algebra; we discuss possible representation-theoretic consequences of this fact. 
Our results, respectively conjectures, imply the
modularity of the generating function for surfaces with type $A$ and
type $D$, respectively arbitrary, simple singularities, confirming
predictions of $S$-duality. 
\end{abstract}

\maketitle


\section{Euler characteristics of Hilbert schemes of points}

Let $X$ be a quasiprojective variety $X$ over the field $\SC$ of complex numbers. Let $\mathrm{Hilb}^m(X)$ denote the Hilbert 
scheme of $m$ points on $X$, the quasiprojective scheme parametrizing $0$-dimensional subschemes of $X$ of
length~$m$. 
Consider the generating series of topological Euler characteristics  
\[ Z_X(q)=\sum_{m=0}^\infty \chi\left(\mathrm{Hilb}^m(X)\right)q^m.\] 
For a smooth variety $X$, the series $Z_X(q)$, as well as various refinements, have been extensively studied. 
For a nonsingular curve $X=C$, we have MacDonald's result~\cite{macdonald1962poincare}
\[ Z_C(q)= (1-q)^{-\chi(C)}.
\]
For a nonsingular surface $X=S$, we have (a specialization of) G\"ottsche's formula~\cite{gottsche1990betti}
\begin{equation}\label{eq:goettsche} Z_S(q)=\left(\prod_{m=1}^{\infty}(1-q^m)^{-1}\right)^{\chi(S)}. 
\end{equation}
There are also results for higher-dimensional varieties~\cite{cheah1996cohomology}. 

For singular varieties~$X$, very little is known about the series $Z_X(q)$. For a singular curve $X=C$ with a finite 
set $\{P_1, \ldots, P_k\}$ of planar singularities however, we have the beautiful conjecture of 
Oblomkov and Shende~\cite{oblomkov2012hilbert}, proved by Maulik~\cite{maulik2012stable}, which specializes to the following:
\begin{equation}\label{formula:singcurve}
Z_C(q)= (1-q)^{-\chi(C)}\prod_{j=1}^k Z^{(P_i, C)}(q).
\end{equation}
Here each $Z^{(P_i, C)}(q)$ is a highly nontrivial local term that can be expressed in terms of the HOMFLY polynomial of
the embedded link of the singularity $P_i\in C$.

\section{Simple surface singularities} 

In this announcement, we consider the generating series $Z_S(q)$ for $X=S$ a singular surface with simple (Kleinian, rational 
double point) singularities. First, we discuss the local situation. As is well known, locally analytically~$S$ is a quotient 
singularity $S=\SC^2/G_\Delta$. Here $G_\Delta<\mathrm{SL}(2,\SC)$ is a finite subgroup corresponding to an irreducible simply-laced Dynkin diagram $\Delta$, the dual graph of the
exceptional components in the minimal resolution of the singularity. There are three possible types: 
$\Delta$ can be of type $A_n$ for $n\geq 1$, type $D_n$ for $n\geq 4$ and type $E_n$ for $n=6,7,8$. The following is our
main conjecture. 

\begin{conjecture} Let $\SC^2/G_\Delta$ be a simple singularity associated with the irreducible simply-laced Dynkin
diagram~$\Delta$  with $n$ nodes. 
Let $C_\Delta$ be the Cartan matrix corresponding to $\Delta$, and let $h^\vee$ be the (dual) 
Coxeter number of the corresponding finite-dimensional Lie algebra. Then 
\begin{equation}\label{main:formula} Z_{\SC^2/G_\Delta}(q)=\left(\prod_{m=1}^{\infty}(1-q^m)^{-1}\right)^{n+1}\cdot\sum_{ \overline{m}=(m_1,\dots,m_n) \in \SZ^n } \zeta^{m_1+m_2+ \dots +m_n}(q^{1/2})^{\overline{m}^\top \cdot C_\Delta \cdot \overline{m}},\end{equation}
where $\zeta=\exp({\frac{2 \pi i}{1+h^\vee}})$.
\label{conj:formula}
\end{conjecture}

Evidence for the conjecture is presented in the following result. 

\begin{theorem}
\label{thm:typeAD}
Let $S=\SC^2/G_\Delta$ be a simple singularity with $\Delta$ of type $A_n$ for $n \geq 1$ or $D_n$ for $n \geq 4$. Then 
Conjecture~\ref{conj:formula} holds.
\end{theorem}
\begin{proof} For type $A$, straightforward torus localization leads to a combinatorial series. This series was
determined in closed form by Dijkgraaf and Sulkowski~\cite{dijkgraaf2008instantons} using a direct method, 
obtaining formula~\cite[(10)]{dijkgraaf2008instantons}, which is equivalent to~\eqref{main:formula} for type~$A$. The same result, still for type $A$, was recently re-proved, using wall-crossing in Donaldson--Thomas theory, by Toda~\cite{toda2013s}. We give 
a third, elementary combinatorial argument in~\cite[Section 2]{gyenge2015main}. 

For type $D$, formula~\eqref{main:formula} is the main result of~\cite{gyenge2015main}. 
\end{proof}

\begin{remark} A priori it is not clear at all (at least to us) that the right hand 
side of~\eqref{main:formula} defines an integer
series. We checked numerically that the series has integer coefficients for $\Delta$ of type 
$E_6$, $E_7$ and $E_8$ as well to a high power in~$q$. See below for further discussion. We also note that while Theorem~\ref{thm:typeAD} covers in some sense ``most'' of the cases, the existing proofs are specific to types $A$ and $D$ and it seems difficult to push them through in type $E$.
\end{remark}

\section{Surfaces with simple singularities}

Let $X=S$ be a quasi-projective surface which is non-singular outside a finite number of simple surface singularities $\{P_1, \ldots, P_k\}$, with $(P_i\in S)$ a singularity locally analytically isomorphic to $(0\in \SC^2/G_{\Delta_i})$ for $G_{\Delta_i}<\mathrm{SL}(2,\SC)$ a finite subgroup as above. Let $S^0\subset S$ be the nonsingular part of $S$.

\begin{theorem} \label{thm:singsurface} The generating function $Z_S(q)$
 of the Euler characteristics of Hilbert schemes of points of $S$ has
 a product decomposition
\begin{equation}\label{formula:singsurface}
Z_S(q)= \left(\prod_{m=1}^{\infty}(1-q^m)^{-1}\right)^{\chi(S^0)}\cdot \prod_{j=1}^k Z^{(P_i, S)}(q).
\end{equation}
The local terms can be expressed as
\begin{equation} Z^{(P_i, S)}(q) = Z_{\SC^2/G_{\Delta_i}}(q)\label{eq:localterms}
\end{equation}
and are given by formula~\eqref{main:formula} for $P_i\in S$ of type $A$ and $D$, and, assuming Conjecture~\ref{conj:formula}, also of type $E$.
\end{theorem}
\begin{proof} The product decomposition \eqref{formula:singsurface}, as well as the equality~\eqref{eq:localterms}, follow from a standard argument; we sketch the details. For a point $P\in S$ on a quasiprojective surface~$S$, let $\mathrm{Hilb}_P^{m}(S)$ denote the punctual Hilbert scheme of~$S$ at $P$, the Hilbert scheme of length~$m$ subschemes of~$S$ set-theoretically supported at the single point~$P$. Then for our surface~$S$, we
have 
\[ \chi(\mathrm{Hilb}_{P_i}^{m}(S)) = \chi(\mathrm{Hilb}_0^{m}(\SC^2/G_{\Delta_i}))  = \chi(\mathrm{Hilb}^{m}(\SC^2/G_{\Delta_i})).
\]
Here the first equality follows from the analytic isomorphism between $(P_i\in S)$ and $(0\in \SC^2/G_{\Delta_i})$. The second equality follows from torus localization, using the fact that each singularity $\SC^2/G_{\Delta_i}$ is weighted homogeneous, admitting a retracting $\SC^*$-action which retracts finite subschemes to the origin. On the other hand,  we have the decomposition 
\[ \mathrm{Hilb}^{m}(S)= \bigsqcup_{\sum_{i=0}^k m_i=m} \mathrm{Hilb}^{m_0}(S^0) \times \prod_{i=1}^k \mathrm{Hilb}_{P_i}^{m_i}(S).
\]
Reinterpreting this equality for generating series proves \eqref{formula:singsurface}-\eqref{eq:localterms}, recalling 
also~\eqref{eq:goettsche} for $S^0$. The last part of the statement is Theorem~\ref{thm:typeAD}. 
\end{proof}

Formulas~\eqref{main:formula}-\eqref{formula:singsurface}-\eqref{eq:localterms}
are our analogue for the case of surfaces with simple singularities of the Oblomkov--Shende--Maulik formula~\eqref{formula:singcurve}. Note that each $\SC^2/G_{\Delta_i}$ is in particular a hypersurface singularity, as are planar singularities in the curve case. The main difference with formula~\eqref{formula:singcurve} is the fact that (conjecturally, for type $E$) our local terms $Z^{(P_i, S)}(q)$ are expressed in terms of Lie-theoretic and not topological data. We leave the question whether our local terms have any interpretation of in terms of the topology of the (embedded) link of $P_i$, and whether there are nice formulas for other two-dimensional hypersurface singularities, for further work. 

\section{Modularity results}

Our formulae lead to the following new modularity results, extending the results of~\cite{toda2013s} for type~$A$.

\begin{corollary}\! (S-duality for simple singularities) \ For type $A$ and type $D$, and, assuming Conjecture~\ref{conj:formula}, for all types, the partition function $Z_{\SC^2/G_\Delta}(q)$  is, up to a suitable fractional power of~$q$, the $q$-expansion of a meromorphic modular form of weight $-\frac{1}{2}$ for some congruence subgroup of $\mathrm{SL}(2,\SZ)$.
\label{cor:Sdual}
\end{corollary}
\begin{proof} This follows straight from \cite[Prop.3.2]{toda2013s}.
\end{proof}

\begin{corollary}\! (S-duality for surfaces with simple singularities) \ Let $S$ be a quasiprojective
surface with simple singularities of type $A$ and $D$, or, 
assuming Conjecture~\ref{conj:formula}, of arbitrary type. Then the generating function $Z_S(q)$
is, up to a suitable fractional power of~$q$, the $q$-expansion of a meromorphic modular form 
of weight $-\frac{\chi(S)}{2}$ for some congruence subgroup of $\mathrm{SL}(2,\SZ)$.
\end{corollary}
\begin{proof}  Combine Theorem~\ref{thm:singsurface} and Corollary~\ref{cor:Sdual}. 
\end{proof}

\section{A representation of an affine Lie algebra}
 
Our formula~\eqref{main:formula} has a strong representation-theoretic flavour. We discuss the relevant ideas 
in this section, and conclude with some speculations. 

Let $\frakg_\Delta$ be the complex finite dimensional simple Lie algebra of
rank $n$ corresponding to the irreducible simply-laced Dynkin diagram
$\Delta$ with $n$ nodes. Attached to $\Delta$ is also an (untwisted) affine Lie 
algebra $\widetilde\frakg_\Delta$; a slight variant will be more interesting for us, 
see e.g.~\cite[Sect 6]{etingof2012symplectic}. 
Denote by $\widetilde{\frakg_\Delta\oplus\SC}$ the  Lie algebra that is the direct sum of the affine Lie algebra $\widetilde\frakg_\Delta$ and an infinite Heisenberg algebra $\frakheis$, with their centers identified; we may call $\widetilde{\frakg_\Delta\oplus\SC}$ the extended affine Lie algebra associated with $\Delta$. 

Let $V_0$ be the basic representation of $\widetilde\frakg_\Delta$, the level-1 representation with highest 
weight $\omega_0$, the fundamental weight corresponding to the additional node
of the affine Dynkin diagram. Let $\CF$ be the standard
Fock space representation of $\frakheis$, having central charge 1. Then $V=V_0\otimes \CF$ is a representation of 
$\widetilde{\frakg_\Delta\oplus\SC}$ that we may call the extended basic representation. 

\begin{example} For $\Delta$ of type $A_n$, we have $\frakg_\Delta={\mathfrak{sl}}_{n+1}$, $\tilde\frakg_\Delta=\widetilde{\mathfrak{sl}}_{n+1}$, $\widetilde{\frakg_\Delta\oplus\SC}=\widetilde{\mathfrak{gl}}_{n+1}$.
In this case there is in fact a natural vector space isomorphism $V\cong \CF$ with Fock space itself, see e.g.~\cite[Section 3E]{tingley2011notes}.
\end{example} 

By the Frenkel--Kac theorem~\cite{frenkel1980basic} there is an isomorphism
\[
V \cong \CF^{n+1}\otimes \SC[Q_\Delta],
\]
where $(Q_\Delta, \langle\  \rangle)$ is the root lattice corresponding to the root system $\Delta$. 
Here, for $\beta\in Q_\Delta$, 
$\CF^{n+1}\otimes e^\beta$ is the sum of weight subspaces of weight $\omega_0 - \left(m+\frac{\langle \beta,\beta\rangle}{2}\right)\delta + \beta$, $m \geq 0$, where $\delta$ is the imaginary root. 
Thus, we can write the character of the representation $V$ as
\begin{equation} 
\label{eq:extcharformula}
\mathrm{char}_V(q_0, \ldots, q_n) = e^{\omega_0} \left(\prod_{m>0}(1-q^m)^{-1}\right)^{n+1} \cdot \sum_{ \beta \in Q_\Delta } q_1^{\beta_1}\cdot\dots\cdot q_n^{\beta_n}(q^{1/2})^{\langle \beta,\beta\rangle},
\end{equation}
where $q=e^{-\delta}$, and
$\beta=(\beta_1, \ldots, \beta_n)\in Q_\Delta$ is the expression of an
element of the root lattice $Q_\Delta$ in terms of the simple roots. 

\begin{remark} The representation $V$ itself is of course well known to appear in the equivariant geometry
of the pair $(\SC^2, G_\Delta)$. Given any finite-dimensional representation $\rho\in{\rm Rep}(G_\Delta)$, 
denote by ${\rm Hilb}^\rho(\SC^2)$ the
$\rho$-Hilbert scheme of $\SC^2$, the scheme representing $G_\Delta$-equivariant ideals $I\lhd\SC[\SC^2]$ with
$\SC[\SC^2]/I\cong\rho$. These are all Nakajima quiver varieties for the affine Dynkin diagram associated with
$\Delta$ (the McKay quiver of $G_\Delta$). By results of Nakajima~\cite{nakajima1994instantons, nakajima2002geometric},
there is in fact an isomorphism of
$\widetilde{\frakg_\Delta\oplus\SC}$-representations
\[ V\cong \bigoplus_{\rho\in{\rm Rep}(G_\Delta)} H^*({\rm Hilb}^\rho(\SC^2)),
\]
where the $\widetilde{\frakg_\Delta\oplus\SC}$-action on the right hand side
is given by natural correspondences. Since these varieties have no odd
cohomology~\cite[Section 7]{nakajima2001quiver}, the Euler characteristic generating series of the right
hand side is given precisely by the
character~\eqref{eq:extcharformula} of $V$.
\end{remark}

Comparing formulas \eqref{main:formula} and~\eqref{eq:extcharformula}, we deduce the following. 

\begin{theorem} For $\Delta$ of type $A$ and $D$, or, assuming Conjecture~\ref{conj:formula}, of 
arbitrary type, the generating 
series $Z_{\SC^2/G_\Delta}(q)$ is a specialization of the character formula of the extended basic representation 
of the associated extended affine Lie algebra $\widetilde{\frakg_\Delta\oplus\SC}$, obtained by setting $q_i=\exp({\frac{2 \pi i}{1+h^\vee}})$ 
for $i=1, \ldots, n$. 
\end{theorem}

We remark here that Dijkgraaf and Sulkowski~\cite{dijkgraaf2008instantons} already noticed
the fact that~\eqref{main:formula} for type $A$ is a specialized character formula for~$\widetilde{\mathfrak{sl}}_{n+1}$.

We do not know whether~\eqref{main:formula} itself is a character
formula for a Lie algebra, but this appears plausible. Such a result would
of course imply in particular that its coefficients are (positive)
integers. It would also suggest that suitable cohomology groups $\oplus_m H^*(\mathrm{Hilb}^m(\SC^2/G_\Delta))$, 
and indeed  $\oplus_m H^*(\mathrm{Hilb}^m(S))$ for $S$ a surface with simple singularities, would carry actions of 
interesting Lie algebras, generalizing the Heisenberg algebra actions of Grojnowski and 
Nakajima~\cite{grojnowski1996instantons,nakajima1997heisenberg} when $S$ is a smooth surface. 

In a different direction, one may also wonder about a higher-rank generalization. 
Our current discussion involves Hilbert schemes, parametrizing rank $r=1$ sheaves on the singular surface. 
In the relationship between the instantons on algebraic surfaces and affine Lie algebras, 
level equals rank~\cite{grojnowski1996instantons}. Indeed the extended basic represenatation $V$ has level $l=1$. 
Thus the substitution above is by the root of unity $\zeta=\exp({\frac{2 \pi i}{l+h^\vee}})$. In higher rank, one may wonder whether
there are similar formulae involving Euler characteristics of 
degenerate versions of the moduli space of rank $r=l$ framed $G_\Delta$-equivariant vector bundles on $\SP^2$ on the one hand, 
and specialized character formulae of level $l$ representations of affine Lie algebras on the other. 

To conclude, we point out that the substitution by the root of unity $\zeta=\exp({\frac{2 \pi i}{l+h^\vee}})$ appears 
elsewhere in representation theory, notably in the Verlinde formula, and in the Kazhdan--Lusztig
equivalence between certain categories of representations of the affine Lie algebra, respectively of the finite
quantum group. We are unaware of any connection between our work and those circles of ideas. 

\subsection*{Acknowledgements} The authors would like to thank Gwyn
Bellamy, Alexander Braverman, Alastair Craw, Eugene Gorsky, Ian
Grojnowski, Kevin McGerty, Iain Gordon, Tomas Nevins, Travis Schendler and Tam\'as Szamuely for helpful comments and discussions. \'A.Gy.~was partially supported by the \emph{Lend\"ulet program} (Momentum Programme) of the Hungarian Academy of Sciences and by ERC Advanced Grant LDTBud (awarded to Andr\'as Stipsicz). A.N.~was partially supported by OTKA Grants 100796 and K112735. B.Sz.~was partially supported by EPSRC Programme Grant EP/I033343/1.

\bibliographystyle{amsplain}
\bibliography{announcement}

\providecommand{\bysame}{\leavevmode\hbox to3em{\hrulefill}\thinspace}
\providecommand{\MR}{\relax\ifhmode\unskip\space\fi MR }
\providecommand{\MRhref}[2]{%
  \href{http://www.ams.org/mathscinet-getitem?mr=#1}{#2}
}
\providecommand{\href}[2]{#2}
\begin{thebibliography}{10}

\bibitem{cheah1996cohomology}
J.~Cheah, \emph{{On the cohomology of Hilbert schemes of points}}, J. Algebraic
  Geom. \textbf{5} (1996), no.~3, 479--512.

\bibitem{dijkgraaf2008instantons}
R.~Dijkgraaf and P.~Su{\l}kowski, \emph{{Instantons on ALE spaces and orbifold
  partitions}}, J. High Energy Phys. \textbf{3} (2008), 013--013.

\bibitem{etingof2012symplectic}
P.~Etingof, \emph{{Symplectic reflection algebras and affine Lie algebras}},
  Mosc. Math. J \textbf{12} (2012), no.~3, 543--565.

\bibitem{frenkel1980basic}
I.~B. Frenkel and V.~G. Kac, \emph{{Basic representations of affine Lie
  algebras and dual resonance models}}, Invent. Math. \textbf{62} (1980),
  no.~1, 23--66.

\bibitem{gottsche1990betti}
L.~G{\"o}ttsche, \emph{{The Betti numbers of the Hilbert scheme of points on a
  smooth projective surface}}, Math. Ann. \textbf{286} (1990), no.~1, 193--207.

\bibitem{grojnowski1996instantons}
I.~Grojnowski, \emph{{Instantons and affine algebras I: The Hilbert scheme and
  vertex operators}}, Math. Res. Lett. \textbf{3} (1996), no.~2, 275--292.

\bibitem{gyenge2015main}
\'A. Gyenge, A.~N\'emethi, and B.~Szendrői, \emph{{Euler characteristics of
  Hilbert schemes of points on simple surface singularities}},
  \url{http://arxiv.org/}, 2015.

\bibitem{macdonald1962poincare}
I.~G. Macdonald, \emph{{The Poincar\'e polynomial of a symmetric product}},
  Math. Proc. Cambridge Philos. Soc. \textbf{58} (1962), no.~4, 563--568.

\bibitem{maulik2012stable}
D.~Maulik, \emph{{Stable pairs and the {HOMFLY} polynomial}},
  \url{http://arxiv.org/abs/1210.6323}, 2012.

\bibitem{nakajima1994instantons}
H.~Nakajima, \emph{{Instantons on ALE spaces, quiver varieties, and Kac-Moody
  algebras}}, Duke Math. J. \textbf{76} (1994), no.~2, 365--416.

\bibitem{nakajima1997heisenberg}
\bysame, \emph{{Heisenberg algebra and Hilbert schemes of points on projective
  surfaces}}, Ann. of Math. (2) \textbf{145} (1997), no.~2, 379--388.

\bibitem{nakajima2001quiver}
\bysame, \emph{{Quiver varieties and finite dimensional representations of
  quantum affine algebras}}, J. Amer. Math. Soc. \textbf{14} (2001), no.~1,
  145--238.

\bibitem{nakajima2002geometric}
\bysame, \emph{{Geometric construction of representations of affine algebras}},
  Proceedings of the International Congress of Mathematicians (Beijing, 2002),
  vol.~1, IMU, Higher Ed, 2002, pp.~423--438.

\bibitem{oblomkov2012hilbert}
A.~Oblomkov and V.~Shende, \emph{{The Hilbert scheme of a plane curve
  singularity and the HOMFLY polynomial of its link}}, Duke Math. J.
  \textbf{161} (2012), no.~7, 1277--1303.

\bibitem{tingley2011notes}
P.~Tingley, \emph{{Notes on Fock Space}},
  \url{http://webpages.math.luc.edu/~ptingley/}, 2011.

\bibitem{toda2013s}
Y.~Toda, \emph{{S-Duality for surfaces with $A_n$-type singularities}}, Math.
  Ann. \textbf{363} (2015), 679--699.

\end{thebibliography}

\end{document}